\documentclass[11pt]{article}
\usepackage{amsmath}
\usepackage{amssymb}
\usepackage{amsthm}
\usepackage[ansinew]{inputenc}
\usepackage{eufrak}
\usepackage{euscript}
\usepackage{amscd}
\usepackage{lscape}

\usepackage[all]{xy}
\usepackage{url}
\usepackage{graphicx}

\usepackage{verbatim}

\def\color[#1]#2{}

\makeatletter \gdef\th@break{\normalfont\slshape
  \def\@begintheorem##1##2{\item[%
       \rlap{\vbox{\hbox{\hskip \labelsep\theorem@headerfont ##1\ ##2}%
                   \hbox{\strut}}}]}%
\def\@opargbegintheorem##1##2##3{%
  \item[\rlap{\vbox{\hbox{\hskip \labelsep \theorem@headerfont
                     ##1\ ##2\ ##3}%
                    \hbox{\strut}}}]}}
\makeatother

\newtheorem{theorem}{Theorem}[section]

\newtheorem{conjecture}{Conjecture}[section]

\newtheorem{proposition}{Proposition}[section]
\newtheorem{corollary}{Corollary}[section]
\newtheorem{remark}{Remark}

\newtheorem{example}{Example}
\newtheorem{lemma}{Lemma}[section]

\newcommand{\bfR}{\mathbf{R}}
\newcommand{\bfS}{\mathbf{S}}
\newcommand{\bfT}{\mathbf{T}}
\newcommand{\Ch}[2]{\begin{bmatrix} #1 \\ #2 \end{bmatrix}}

\newcommand{\epsm}{\boldsymbol{\varepsilon}}

\newcommand{\lm}{\boldsymbol{\lambda}}
\newcommand{\wm}{\boldsymbol{\omega}}

\newcommand{\Am}{\mathsf{A}}
\newcommand{\Cm}{\mathsf{C}}

\newcommand{\Mm}{\mathsf{M}}

\newcommand{\Vm}{\mathsf{V}}

\newcommand{\calL}{\mathcal{L}}
\newcommand{\calO}{\mathcal{O}}

\newcommand{\sM}{\mathsf{M}}
\def\GL{\mathsf{GL}}

\def\Q{\mathbb{Q}}
\def\Z{\mathbb{Z}}
\def\C{\mathbb{C}}

\def\R{\mathbb{R}}
\def\H{\mathbb{H}}

\def\GF{\mathbb{F}}
\def\Sp{\textrm{Sp}}

\def\ov{\overline}

\def\Jac{\textrm{Jac}}
\def\Aut{\textrm{Aut}}

\def\Hom{\textrm{Hom}}

\def\End{\textrm{End}}

\def\Gal{\textrm{Gal}}

\def\Pic{\textrm{Pic}}

\def\im{\textrm{Im}}

\def\re{\textrm{Re}}

\title{Explicit computations of Serre's obstruction for genus $3$ curves and application to optimal curves}
\author{Christophe Ritzenthaler\footnote{partially supported by grant MTM2006-11391 from the Spanish MEC}}

\begin{document}
\maketitle

\begin{abstract}
Let $k$ be a field of characteristic different from $2$. There can be an obstruction for a principally polarized abelian threefold $(A,a)$ over $k$, which is a Jacobian over $\overline{k}$, to be a Jacobian over $k$. It can be computed  in terms of the rationality of the square root  of the value of a certain Siegel modular form.  
We show how to do this explicitly   for principally polarized abelian threefolds which are the third power of an elliptic curve with complex multiplication. We use our numeric results to prove or refute the existence of some optimal curves of genus $3$. 
\end{abstract}

\section{Introduction}
Let $p$ be a prime and $q=p^n$ for $n>0$. For $g \geq 0$, it is well known that the maximal number of points of a (smooth, absolutely irreducible, projective) curve of genus $g$ over $\GF_q$ is less than or equal to  Serre-Weil bound $q+1+g m$ where $m=\lfloor 2 \sqrt{q} \rfloor$. However, it is a long standing problem to give the precise value of this maximum, denoted $N_q(g)$. Even to know if, for given $q$ and $g$, there is a a genus $g$ curve over $\GF_q$ (called \emph{optimal}) which reaches Serre-Weil bound is a difficult problem. Closed formulas for any $q$ are only known for $g=1$ \cite{deuring}, \cite{waterhouse} and $g=2$ \cite{serreg2}. As soon as  $g \geq 3$, one only knows $N_q(g)$ for some  $q$  square or for small $q$  (see  \cite{ibukiyama}, \cite{NR} and \cite{top} for $g=3$  and \cite{vandergeer}  for $g \leq 50$).\\
Still, for $g=3$,  there is a general result due to Lauter \cite{lauterg3} who proves  that for all $q$ there  exists a genus $3$ curve $C$ over $\GF_q$ such that
\begin{equation} \label{pm3}
| \# C(\GF_q) -q-1| \geq 3m-3.
\end{equation}
Hence there is either a curve whose number of points is up to $3$ from Serre-Weil bound \emph{or} up to $3$ from the minimum number of points $q+1-3m$. This ambiguity, which prevents from getting a quasi-optimal result, is actually a consequence of a  general theorem  derived from a precise form of Torelli's theorem. We give here a version due to Serre.
\begin{theorem}[{\cite[Appendix]{lauter}}]
\label{twist}
Let $(A,a)$ be a principally polarized abelian variety  of dimension $g \geq 1$
over a field $k$. Assume that $(A,a)$ is isomorphic over $\overline{k}$ to
the Jacobian of a curve $C_{0}$ of genus $g$ defined over $\overline{k}$.
The following alternative holds :
\begin{enumerate}
\item
\label{TWI1}
If $C_{0}$ is hyperelliptic, there is a curve $C/k$ isomorphic to
$C_{0}$ over $\overline{k}$ such that $(A,a)$ is $k$-isomorphic
to $(\Jac(C),j)$ where $j$ is the canonical polarization on $\Jac(C)$.
\item
\label{TWI2}
If $C_{0}$ is not hyperelliptic, there is  a curve $C/k$ isomorphic to
$C_{0}$ over $\overline{k}$, and a quadratic character
$$
\begin{CD}
\epsilon : \Gal(k^{\textrm{sep}}/k) & @>>> &
\{\pm 1\}
\end{CD}
$$
such that the twisted abelian variety
$(A,a)_{\epsilon}$ is $k$-isomorphic to $(\Jac(C),j)$. The character
$\epsilon$ is trivial if and only if $(A, a)$ is $k$-isomorphic to a
Jacobian.
\end{enumerate}
\end{theorem}

Let us explain how this applies when $k=\GF_q$. In order to prove the existence of a curve with a certain number of points $N$, one strategy is to construct an abelian variety over $k$ whose Frobenius endomorphism $\pi$ has trace $1+q-N$ and then to prove that this abelian variety is actually a Jacobian. When $g<4$, this is tractable because every abelian variety with an (indecomposable) polarization is isomorphic over $\overline{k}$ to a Jacobian (\cite{Hoyt}, \cite{Ueno}, see Th.\ref{hu}). Any genus $2$ curve is hyperelliptic, so there is no  obstruction and  one is able to describe exactly which isogeny classes of abelian surfaces over $k$ contain a Jacobian (see \cite{HNR}). However, as Th.\ref{twist} shows, as soon as $g=3$,  curves may be non hyperelliptic and there may be an obstruction to be a Jacobian over $k$. Moreover if it is the quadratic twist which is a Jacobian, then the corresponding trace is $-(1+q-N)$ and this explains the absolute value in \eqref{pm3}.\\

Serre did not only state this obstruction, he also suggested a strategy to compute it when $g=3$ and $k$ is a field of characteristic different from $2$ (see \cite[Letter to Top]{LR}). Roughly speaking, an absolutely  indecomposable principal polarized abelian threefold $(A,a)$ over $k$ is a Jacobian if the value of a certain Siegel modular form $\chi_{18}$ at the `moduli point' $(A,a)$, with respect to a rational basis of regular differentials, is a square in $k$. This assertion was  motivated by a formula of Klein \cite{klein} relating $\chi_{18}$  to the square of the discriminant of plane quartics, which  more or less gives the `only if' part of the claim.  In \cite{LRZ}, Serre's assertion was proved when $k$ is a number field.\\

In the present article, we go back to the initial motivation of Serre's letter, i.e. the existence of optimal curves of genus $3$. In Sec.\ref{secoptimal}, we  show that the ideas developed in \cite{LRZ} can be applied over finite fields as well and that the reduction of $\chi_{18}$ still represents the obstruction (Prop.\ref{chi18F}). For a different approach, see \cite{meagher}. However, one does not know how to compute directly this value over finite fields. As Serre suggested, we lift $(A,a)$ over a number field and there we can use the analytic  expression of $\chi_{18}$ in terms of Thetanullwerte. Doing the computation with enough precision, we can recognize this value as an algebraic number. Finally we reduce it to the finite field to obtain the obstruction (see Prop.\ref{mainresult} for a more precise formulation).

As the Jacobian of an optimal curves is isogenous to the power of an elliptic curve $E$, we make this procedure explicit in the particular case where $A=E^3$. Let $a_0$ be the product principal polarization on $E^3$ and $M=a_0^{-1} a \in M_3(\End(E))$. When $\End(E)$ is an order in an imaginary quadratic field, it is well known that $M$ is the matrix of a principal polarization on $E^3$ if and only if $M$ is a positive definite hermitian matrix with determinant $1$ (Prop.\ref{pp}). In  Sec.\ref{pola}, when $E$ is defined over a number field,  we  show how to translate the data $(E^3,a_0 M)$ into a specific period matrix of the corresponding torus in order to start the analytic computation of $\chi_{18}$. 

For certain orders, the set of matrices $M$ (up to isomorphisms) has been explicitly computed (see  \cite{schiemann}).  This enables us to give in Sec.\ref{secexplicit} tables of values of $\chi_{18}$ on $E^3$ for certain CM elliptic curves $E$ with class number $1$. As an application, we show in Cor.\ref{someoptimal} that there is an optimal genus $3$ curve over $\GF_q$ for $q=47,61,137,277, 23^3$ but not for $q=311$. Note that this result for $q=47$ and $q=61$ has already been obtained in \cite{top} using explicit models.


Besides showing some applications of Serre's strategy, the purpose of this article is to present data to support future conjectures on the primes dividing $\chi_{18}$ and their exponent. Indeed, if analytic computations may work out the value of $N_q(3)$ case by case, a general formula for  $N_q(3)$ will require an algebraic interpretation of $\chi_{18}$.\\

\noindent
{\bf Conventions and notation.}
If $A$ and $B$ are varieties over a field $k$, when we speak of a morphism from $A$ to $B$ we always mean  a morphism \emph{defined over $k$}. So, for instance $\End(A)$ is the ring of endomorphisms defined over $k$, $A \sim B$ means $A$ isogenous to $B$ over $k$, etc. If $(A,a)$ and $(B,b)$ are polarized abelian varieties, by an isomorphism between them, we always mean `as polarized abelian varieties'. \\ 

\noindent
{\bf Acknowledgements.} I would like to thank  Kristin Lauter who inspired this project during her visit in June 2008 and  Gilles Lachaud and Alexey Zykin for helpful discussions during the writing of the first version of this article. I am also grateful to Jordi Guàrdia who computed the model of the curve $X_{19,1}$ for me. Finally I would like to thank Jean-François Mestre and Jean-Pierre Serre for their stimulating comments and questions on the first version of this article.  \\

\section{Serre's problem and optimal genus $3$ curves} \label{secoptimal}
\subsection{Reduction properties}
In the sequel we will recall  the notation and some results of \cite{Ichi3} and \cite{LRZ}. Let $g \geq 2$ be an integer.
Let $\Mm_g:=\Mm_{g,1}$ (resp. $\Am_g:=\Am_{g,1}$)  be the moduli stack of smooth and proper curves of genus $g$ (resp. of principally polarized abelian schemes of relative dimension $g$). Let $\pi : \Cm_g \to \Mm_g$ be the universal curve (resp. $\pi : \Vm_g \to \Am_g$ be the universal abelian scheme) and $\lm=\bigwedge^{g} \pi_{*}\Omega^{1}_{\Cm_{g} / \Mm_{g}}$ (resp. $\wm=\bigwedge^{g} \pi_{*}\Omega^{1}_{\Vm_{g} / \Am_{g}}$) be the Hodge bundle on $\Mm_g$ (resp. $\Am_g$). For any $h \geq 0$ and any $\Z$-algebra $R$ we denote $\bfT_{g,h}(R) = \Gamma(\Mm_{g} \otimes R, \lm^{\otimes h})$ (resp. 
$\bfS_{g,h}(R) = \Gamma(\Am_{g} \otimes R,\wm^{\otimes h})$)
the $R$-module of \emph{Teichmüller modular forms} (resp. \emph{geometric Siegel modular forms}). If $C \in \Mm_g \otimes R$ (resp. $(A,a) \in \Am_g \otimes R$) and $f \in \bfT_{g,h}(R)$ (resp. $f \in \bfS_{g,h}(R)$) then for any basis of regular differentials $\omega_1,\ldots,\omega_g$ of $\Omega^1[C] \otimes R$ (resp. $\Omega^1[A] \otimes R$) we can define an element
$$f(C,\lambda)=f(C)/\lambda^{\otimes h} \in R \; (\textrm{resp.} \;   f((A,a),\omega)=f(A,a)/\omega^{\otimes h} \in R)$$
where $\lambda=\omega_1 \wedge \cdots \wedge \omega_g \in \lm \otimes R$ (resp. $\omega=\omega_1 \wedge \cdots \wedge \omega_g \in \wm \otimes R$).
\begin{lemma}[{\cite[p.138(i)]{FC}}] \label{association}
 The association
$$C \mapsto f(C,\omega) \; (\textrm{resp.} \; (A,a) \mapsto f((A,a),\omega))$$ is compatible with arbitrary pull-backs and isomorphisms.
\end{lemma}
We now assume that $R=k \subset \C$ is a number field. As usually, let us denote  $$\H_g=\{ \tau \in \GL_g(\C), \; \tau \, \textrm{symmetric and} \, \im(\tau)>0\}$$ and let $\widetilde{f}$ belong to the usual vector space $\mathbf{R}_{g,h}(\C)$
of \emph{analytic Siegel modular forms} of weight $h$ on $\H_{g}$, consisting of complex holomorphic functions $\phi(\tau)$ on
$\H_{g}$ satisfying
$$
\phi(M.\tau) = \det (\gamma \tau+ \delta)^{h} \phi(\tau)
$$
for any $M=\left(\begin{array}{cc} \alpha & \beta \\ \gamma & \delta \end{array}\right)$ in the symplectic group $\Sp_{2g}(\Z)$. 
Let  $(A,a)$ be a principally polarized abelian variety of dimension $g$
defined over $k$.
Let $\omega_{1}, \dots, \omega_{g}$ be a basis of $\Omega^{1}[A] \otimes k$
and $\omega = \omega_{1} \wedge \dots \wedge \omega_{g} \in \wm[A] \otimes k$. Let
$\gamma_{1}, \dots \gamma_{2g}$ be a symplectic basis for the polarization
$a$. We say that the matrix
$$
\Omega_a =
\begin{pmatrix}
\int_{\gamma_{1}}\omega_{1} & \cdots & \int_{\gamma_{2g}}\omega_{1}\\
\vdots &  & \vdots \\
\int_{\gamma_{1}}\omega_{g} & \cdots & \int_{\gamma_{2g}}\omega_{g}\\
\end{pmatrix}
$$
is \emph{a period matrix associated to $a$}.  If we write $\Omega_a=[\Omega_1,\Omega_2]$ with $\Omega_i \in M_g(\C)$, Riemann conditions (\cite[p.75]{bl}) imply that $\tau_a=\Omega_2^{-1} \Omega_1 \in \H_g$. We call $\tau_a$ a \emph{Riemann matrix associated to $a$}.
\begin{proposition}[{\cite[Prop.2.4.4]{LRZ}}]
\label{Riemann}
With the previous notation,
$$
f((A,a))=(2i \pi)^{gh}  \frac{\widetilde{f}(\tau_a)}{\det \Omega_{2}^{h}} \omega^{\otimes h} \in \bfS_{g,h}(\C).
$$
Moreover the map $\widetilde{f} \mapsto f$ from $ \mathbf{R}_{g,h}(\C)$ to $\bfS_{g,h}(\C)$ is an isomorphism. 
\end{proposition}

Let $S$ be the localization of the ring of integers $\calO_k$ at a prime 
$\mathfrak{p}$ over a prime $p \ne 2$ and let $F=S/\mathfrak{p}$ be the finite residue field. Let
   $(A,a)$ be a principally polarized abelian variety over $k$ and assume that it has a model over $S$ with good reduction at $\mathfrak{p}$.    We denote this model $(\tilde{A},\tilde{a})$. Let $\omega_1,\ldots,\omega_g$ be a basis of regular differentials of  $\Omega^1[\tilde{A}] \otimes S$ and let $\omega=\omega_1 \wedge \cdots \wedge \omega_g$.  Then by Lem.\ref{association}, for $f \in \bfS_{g,h}(S)$ 
$$f((\tilde{A},\tilde{a}) \otimes F,\omega \otimes F)=f((A,a),\omega \otimes k) \pmod{\mathfrak{p}}.$$
Hence, using Lem.\ref{association} and Prop.\ref{Riemann} we get the following proposition.
\begin{proposition} \label{reductionp}
$$f((\tilde{A},\tilde{a}) \otimes F,\omega \otimes F)=(2i \pi)^{gh}  \frac{\widetilde{f}(\tau_a)}{\det \Omega_{2}^{h}} \pmod{\mathfrak{p}}.$$
\end{proposition}

\subsection{The modular form $\chi_h$}
Following \cite{bl}, we recall the definition of theta functions with (entire)
characteristics
$$
[\epsm] = \Ch{\varepsilon_{1}}{\varepsilon_{2}} \in \Z^g \oplus \Z^g.
$$
 The \emph{(classical) theta function} is given, for
$\tau \in \H_{g}$ and $z \in \C^g$, by
$$
\theta  \Ch{\varepsilon_{1}}{\varepsilon_{2}}(z, \tau) =
\sum_{n \in \Z^g}
\exp(i \pi (n + \varepsilon_{1}/2)  \tau {^t (n + \varepsilon_{1}/2)} + 2 i \pi
(n + \varepsilon_{1}/2){^t (z+\varepsilon_{2}/2)}).
$$
The \emph{Thetanullwerte} are the values at $z = 0$ of these functions.
Recall that a characteristic is \emph{even} if $\varepsilon_1 \cdot {^t \varepsilon_2}
\equiv 0 \pmod{2}$ and \emph{odd} otherwise. 
For $g \geq 2$, we put $h = 2^{g - 2}(2^g + 1)$ and
$$\widetilde{\chi}_h(\tau)=(-1)^{gh/2} \cdot \frac{(2i\pi)^{gh}}{2^{2^{g-1}(2^g-1)}} \cdot  \prod_{\epsm } \theta \Ch{\varepsilon_1}{\varepsilon_2}(0,\tau),$$
$\epsm$ running over the even characteristics with coefficients in $\{0,1\}$.\\

In  \cite{igusa2}, Igusa proves that  if $g \geq 3$, then $\widetilde{\chi}_h(\tau) \in \bfR_{g,h}(\C)$.
Starting from the analytic Siegel modular form $\widetilde{\chi}_{h}$, by Prop.\ref{Riemann},  we can define a geometric Siegel modular form  
$$\chi_{h}((A,a))=(-1)^{gh/2} \cdot \frac{(2i\pi)^{gh}}{2^{2^{g-1}(2^g-1)}} \cdot  \frac{\widetilde{\chi}_{h}(\tau_a)}{\det(\Omega_2)^{h}} (\omega_1 \wedge \cdots \wedge \omega_g)^{\otimes h} \in \bfS_{g,h}(\C).$$ Ichikawa proved in \cite[Prop.3.4]{Ichi3} that, since $\chi_{h}$ has rational integral Fourier coefficients, $\chi_h$ is actually defined over $\Z$, i.e. $\chi_{h} \in \bfS_{g,h}(\Z)$. He also showed that $\chi_h$ is \emph{primitive}, i.e. its reduction modulo any prime is non zero.

\subsection{The case  $g=3$}
We specialize to the case $g=3$, i.e.
\begin{equation} \label{chi18}
\chi_{18}((A,a))=\frac{(2\pi)^{54}}{2^{28}} \cdot  \frac{\widetilde{\chi}_{18}(\tau_a)}{\det(\Omega_2)^{18}} (\omega_1 \wedge \omega_2 \wedge \omega_3)^{\otimes 18} \in \bfS_{3,18}(\Z).
\end{equation}
In \cite{Ichi4}, Ichikawa proves that there exists a Teichmüller modular form $\mu_{3,9}\in \bfT_{3, 9}(\Z)$ of weight $9$ such that if $t : \Mm_3 \to \Am_3$ is the Torelli map then
\begin{equation}
\label{chimu}
t^{*}(\chi_{18}) =   (\mu_{3,9})^2.
\end{equation}

Using this equality we can adapt the proof of \cite[Th.4.2.1]{LRZ} to the case of finite fields.
\begin{proposition} \label{chi18F}
Let $(A,a)$ be a principally polarized abelian threefold defined over a finite field $F$ of characteristic different from $2$. Let $\omega_{1}, \omega_{2}, \omega_{3}$ be a basis of
$\Omega^{1}[A] \otimes F$ and $\omega=\omega_1 \wedge \omega_2 \wedge \omega_3$. Then $(A,a)$ is isomorphic to the Jacobian of a non hyperelliptic genus $3$ curve if and only if $\chi_{18}((A,a),\omega)$
is a non-zero square in $F$. Moreover if $\chi_{18}((A,a),\omega)$ is a non-square in $F$ then $(A,a)$ is not a Jacobian.
\end{proposition}
\begin{proof}
 If $(A,a)$ is isomorphic over $F$ to the Jacobian of a non
hyperelliptic genus $3$ curve $C/F$ then  
 using \eqref{chimu}, we get
$$\chi_{18}((A,a),\omega)= t^*(\chi_{18})(C,\lambda)= \mu_{3,9}^2(C,\lambda)$$
with $\lambda=\theta^* \omega$.  So it is a square. Moreover it is non zero. Indeed by definition of $\mu_{3,9}$ (see \cite[p.1059]{Ichi2}), this form is zero precisely on the hyperelliptic locus.
Now,  assume that  $\chi:=\chi_{18}((A,a),\omega)$ is a non-zero square. By \cite[Lem. 10 and 11]{igusa2}, we know that the modular form $\chi_{18}$ is zero on $(\Am_3 \setminus t(\Mm_3)) \otimes \C$ hence this is true also on the schematic closure and therefore over finite fields. From this, we deduce that $(A,a) \in t(\Mm_3 \otimes \overline{F})$ and since $\chi \ne 0$, we even know that $(A,a)$ is geometrically the Jacobian of a non hyperelliptic genus $3$ curve. Now by Th.\ref{twist}, we know that either  $(A,a)$ is  a Jacobian or  its quadratic twist $(A',a')$ is (and exactly one option holds). Suppose that $(A,a)$ is not a Jacobian. Let  $\omega'_1,\omega'_2,\omega'_3$ be a basis of $\Omega^1[A'] \otimes F$ and $\omega'=\omega_1' \wedge \omega_2' \wedge \omega_3'$.  From what we have just proven $\chi_{18}((A',a'),\omega')$ is a non-zero square in $F$. From \cite[Cor.2.4.3]{LRZ}, we know that
there exists a non square $c \in F$
such that
$${\chi}_{18}((A,a),\omega)=
c^9 \cdot {\chi}_{18}((A',a'),\omega').$$
Hence $\chi$ is not a square and we get a contradiction.
\end{proof}

 This proposition alone is however useless for applications as, so far,  we do not know how to compute directly
$\chi_{18}((A,a),\omega)$. This is why we use Serre's initial idea of lifting a principally polarized abelian threefold  over a local field of characteristic $0$, using the analytic expression
\eqref{chi18} and then reducing the result to get the obstruction by Prop.\ref{reductionp} and \ref{chi18F}. Note that we can always lift a principally polarized abelian threefold. Indeed, by Th.\ref{hu}, we see that, up to a possible quadratic twist, any principally polarized abelian threefold $(A,a)$ over a finite field $F$ is a product of Jacobians $(\Jac(C_i),j_i)$. It is then enough to lift each of the curves $C_i$ to the curve $\tilde{C}_i$ and to consider the principally polarized abelian threefold $\prod (\Jac(\tilde{C}_i),\tilde{j}_i)$.  If necessary we apply a quadratic twist to get a lift of $(A,a)$.
Let us summarize the procedure.

\begin{proposition} \label{mainresult}
Let $(A,a)$ be a principally polarized abelian threefold over a finite field $F$ of characteristic different from $2$ and let $(\tilde{A},\tilde{a})$ be a lift of $(A,a)$ defined over a local ring $S$ of a number field $k$ with residue field $F$. Then for any choice of a basis of regular differentials $\omega_{1}, \omega_2, \omega_{3} \in \Omega^{1}[\tilde{A}] \otimes S$ 
$$\chi:=\chi_{18}((\tilde{A},\tilde{a}),\omega_1 \wedge \omega_2 \wedge \omega_3)=\frac{(2\pi)^{54}}{2^{28}} \cdot  \frac{\widetilde{\chi}_{18}(\tau_{\tilde{a}})}{\det(\Omega_2)^{18}} $$
belongs to $S$. Let $\mathfrak{p}$ be such that $F=S/\mathfrak{p}$. Then
$(A,a)$ is the Jacobian of a non hyperelliptic genus $3$ curve if and only if $\chi \pmod{\mathfrak{p}}$ is a non-zero square in $F$.  Moreover if $\chi \pmod{\mathfrak{p}}$ is a  non-square in $F$ then $(A,a)$ is not a Jacobian. 
\end{proposition}

\section{Polarizations and endomorphisms of abelian varieties } \label{pola}
\subsection{Polarization and symmetric endomorphisms} \label{poen}
Let $(A,a_0)$ be a principally polarized abelian variety of dimension $g>0$ over  a field $k$. There exists an ample line bundle $\calL_0 \in \Pic_{\ov{k}}(A)$ such that $a_0=\phi_{\calL_0} : A \to \hat{A}$. Let $NS(A)$ be the Néron-Severi group of $A$. The map $\calL \mapsto \phi_{\calL}$ from $NS_{\ov{k}}(A)$ to $\Hom_{\ov{k}}(A,\hat{A})$ is injective. If we compose this map with $\phi_{\calL_0}^{-1}$, we then obtain an injection  $NS_{\ov{k}}(A) \hookrightarrow \End_{\ov{k}}(A)$.
\begin{proposition}[{\cite[p.137]{milne}}]
Let $(A,a_0)$ be a principally polarized abelian variety over a field $k$ with all endomorphisms defined over $k$. Let $\dag$ be the Rosati involution induced by $a_0$ (i.e. $b \in \End(A) \mapsto a_0^{-1} \circ \hat{b} \circ a_0 \in \End(A)$). Let $End(A)^s$ be the sub-ring of the endomorphisms which are stable by the Rosati involution $\dag$. Then the map
$$\calL \mapsto a_0^{-1} \circ \phi_{\calL}$$
is a group isomorphism from $NS_{\ov{k}}(A)$ to $\End_{\ov{k}}(A)^s$.
\end{proposition}
In particular, we see that if $\End(A)=\End_{\ov{k}}(A)$ then $NS_{\ov{k}}(A)=NS(A)$.
In what follows we  assume that all endomorphisms of $A$ are defined over $k$.\\

To describe the polarizations on $A$, it is then enough to understand the image of the ample line bundles by the previous map. According to
 \cite[p.209]{mumford}, this image  corresponds to  totally positive elements of the formally real Jordan algebra $NS(A) \otimes \R \simeq \End(A)^s \otimes \R$. With our applications in mind we will make the characterization more explicit for a particular case  in the next section.

\subsection{The case where $A$ is a power of an elliptic curve} \label{Apower}
We restrict ourselves to the case $A = E^g$ where $E : y^2+a_1xy+a_3 y=f(x)$, $a_1,a_3 \in k, f \in k[x]$  is an elliptic curve with all endomorphisms defined over $k$. We assume that $\End(E)$ is an order in an imaginary quadratic field $K$ and we identify it with such an $\calO_K \subset \C$ using its action on the regular differential $\omega_E=dx/(2y+a_1x+a_3)$. Let $a_{E}$ be the principal polarization on $E$ and let $a_0= (a_{E} \times \cdots \times a_{E}) : A \to \hat{A}$ be the product principal polarization on $A$. We denote by $\ov{\phantom{1}}$ the Rosati involution on $\End(E)$ associated to $a_E$. It is equal to the complex conjugation on $\calO_K$ through the identification of $\End(E)$ with $\calO_K$. In the same way, identifying $\End(A)$ with $\sM_g(\calO_K)$, we see that the Rosati involution $\dag$ associated to $a_0$ is $M \mapsto {^t \ov{M}}$. We say that $M \in \sM_g(\calO_K)$ is \emph{hermitian} if  $M={^t \ov{M}}$, i.e. $M \in \End(A)^s$.
\begin{proposition}[{\cite[p.209]{mumford}}] \label{pp}
The morphism $M \mapsto a_0 M : \sM_g(\calO_K) \to NS(A)$ restricts to a bijection between positive definite hermitian matrices of determinant $1$   and principal polarizations on $A$.
\end{proposition}

\begin{remark}
Similar results are used in \cite{langepp} to study the number of principal polarizations on product of elliptic curves over $\C$ without complex multiplication.
\end{remark} 

Recall that for any abelian variety $A$, a polarization $a$ on $A$ is called \emph{decomposable} if there exist two abelian varieties $A_1,A_2$ with polarizations $a_1,a_2$ and an isomorphism $\psi : A \to A_1 \times A_2$ such that $a=\hat{\psi} \circ (a_1 \times a_2) \circ \psi$. One says that $A$ is \emph{absolutely indecomposable} if it is not decomposable over $\ov{k}$.

\begin{proposition} \label{indecomp}
Let $A = E^g$ as before. Let $a$ be a principal polarization on $A$ and $M=a_0^{-1} a \in \sM_g(\calO_K)$. If the polarization $a$ is absolutely indecomposable then $M$ is \emph{indecomposable}, i.e. there is no matrix $P \in \GL_g(\calO_K)$ and two square matrices $M_i \in \GL_{g_i}(\calO_K)$ with $g_1+g_2=g$ such that
\begin{equation} \label{rel_decomp}
M  = {^t  \ov{P}}  \left(\begin{array}{cc} M_1 & 0 \\ 0 & M_2 \end{array}\right) P.
\end{equation}
\end{proposition}
\begin{proof}
Let us assume the existence of such $P,M_1,M_2$ and let $\psi$ be the automorphism of $E^g$ associated to $P$. Let  $a_i={a_0}_{|E^{g_i}} M_i : E^{g_i} \mapsto \hat{E}^{g_i}$. As ${^t \ov{P}}=\psi^{\dag}=a_0^{-1} \hat{\psi} a_0$, relation \eqref{rel_decomp} then becomes
$$M=a_0^{-1} a = a_0^{-1} \hat{\psi} (a_1 \times a_2) \psi,$$
so $a=\hat{\psi}\circ  (a_1 \times a_2) \circ \psi$.
As $M$ is the matrix of a principal polarization, Prop.\ref{pp} shows that $M_i$  also are and so the $a_i$ are principal polarizations on $A_i=E^{g_i}$.  So we see that $a$ is decomposable.\\
\end{proof}
When $g \leq 3$, an indecomposable principal polarization is a necessary and sufficient  condition to be geometrically a Jacobian.
\begin{theorem}[{\cite{Hoyt}, \cite{Ueno}}] \label{hu}
Let $(A,a)$ be a principally polarized abelian variety of dimension $g \leq 3$ over an algebraically closed field $k$. Then $(A,a)$ is a Jacobian if and only if $a$ is indecomposable.
\end{theorem}

When $k$ is a finite field, Serre in \cite[Appendix]{lauterg3} gives a  beautiful description of abelian varieties and their polarizations in the case where $A$ is only \emph{isogenous} to $E^g$ and $E$ is an ordinary elliptic curve for which $\calO_K \simeq \End(E)=\Z[\pi]$ with $\pi$ the Frobenius endomorphism of $E$ over $k$.
 If we denote by $L=\Hom(E,A)$ and $L^*=\Hom_{\ov{\calO_K}}(L,\calO_K)$, a polarization $a$ on $A$ induces  a sesquilinear form $h: L \to L^*$ :
$$\begin{array}{cccccc}
h :  & L=\Hom(E,A) &  \to &\Hom(E,\hat{A}) &\to & L^* \\
& f & \mapsto & a f   & \mapsto& ( g \mapsto a_{E}^{-1} \hat{g} a f)
\end{array}$$
whose associated bilinear form is a positive definite hermitian form $L \times L \to \calO_K$.  
The relation with Prop.\ref{pp} in the case $A=E^g$ is the following. If $M=a_0^{-1} a \in \End(A)$ is the matrix of a polarization given by Prop.\ref{pp} then we can define
$$\begin{array}{cccccccc}
h : & L=\Hom(E,A) & \to & L=\Hom(E,A) &  \to &\Hom(E,\hat{A}) &\to & L^* \\
& f & \mapsto & Mf & \mapsto & a_0 Mf & \mapsto& ( g \mapsto a_{E}^{-1} \hat{g} a_0 M f).
\end{array}$$

Now if we identify $L$ with $\calO_K^g$, we can write $$(a_{E}^{-1} \hat{g} a_0) M f={^t \ov{g}} M f.$$
Thus in this basis, the matrix of the bilinear form associated to $h$ is exactly $M$. So, in the case $A=E^g$ over a finite field, we can confuse both points of view and use the following results.

\begin{lemma}[{\cite[Appendix]{lauterg3}}] \label{onecurve}
Assume that $\calO_K$ is principal. If $A$ is isogenous to $E^g$ then $A$ is actually isomorphic to $E^g$.
\end{lemma}
\begin{lemma}[{\cite[Appendix]{lauterg3}}] \label{inde=inde}
Assume that $\calO_K$ is principal. With the notation of Prop.\ref{indecomp}, $a$ is absolutely indecomposable if and only if $M$ is indecomposable.
\end{lemma}

\subsection{The case $A= E^g$ over $\C$} \label{link}
Let $k \subset \C$. As in Sec.\ref{poen}, we suppose that $E : y^2+a_1 xy+a_3y=f(x)$, $a_1,a_3 \in k$,$f \in k[x]$ is an elliptic curve  with all endomorphisms defined over $k$ and that $\End(E)$ is an order $\calO_K$ in an imaginary quadratic field $K$.\\

Let us choose a basis  $\delta_1,\delta_2$ of $H_1(E,\Z)$ such that $\omega_i=\int_{\delta_i} dx/(2y+a_1x+a_3)$ satisfy  $\tau=\omega_1/\omega_2  \in \H_1$. Let $\Omega_{E}=[\omega_1,\omega_2]$ and $a_E=\phi_{\calL_E}$ the principal polarization on $E$. 
We denote  by $I_g$ the identity matrix of $\C^g$ and 
 $$J_{2g}= \left(\begin{array}{cc} 0 & I_g \\ -I_g & 0 \end{array}\right).$$
 
 Since $a_E$ is given by the intersection product $T=J_2$,  by \cite[Lem.4.2.3]{bl}, its first Chern class is the hermitian form
$$H_{E}:=c_1(\calL_E)=2i(\ov{\Omega_{E}} T^{-1} {^t \Omega_{E}})=
\frac{1}{\im\left(\omega_1 \ov{\omega_2}\right)} I_1.$$

Let now $A=E^g=\C^g/\Omega_0 \Z^{2g}$,
where $$\Omega_0= \left(\begin{array}{cccccc} \omega_1 & & &   \omega_2 & &\\  & \ddots & &  &\ddots &  \\  & & \omega_1 & & & \omega_2 \end{array}\right).$$
Let $a_0=a_E^g=\phi_{\calL_0}$ be the product polarization on $A$. An easy computation shows that
$$H_0=c_1(\calL_0)=
\frac{1}{\im\left(\omega_1 \ov{\omega_2}\right)} I_g.$$

We now consider $M \in \sM_g(\calO_K)$ defining a principal polarization $a=a_0 M=\phi_{\calL}$ on $A$. By \cite[Lem.2.4.5]{bl}
$$H:=c_1(\calL)={^t M} H_0=\frac{1}{\im\left(\omega_1 \ov{\omega_2}\right)} {^t M}.$$
The imaginary part of $H$ is an alternating form with integral values on the lattice $\Lambda=\Omega_0  \Z^{2g}$. Its matrix in the chosen basis of $\Lambda$ is
$$T=\im\left({^t \Omega_0} H \ov{\Omega_0}\right)=\im\left(\begin{array}{cc}
\omega_1 H \ov{\omega_1} & \omega_1 H \ov{\omega_2} \\ \omega_2 H \ov{\omega_1} & \omega_2 H \ov{\omega_2} \end{array}\right).$$
We can make this computation a bit more explicit. In the case where $H$ has real coefficients, we get very easily that
$$T=\left(\begin{array}{cc} 0 & M  \\ -M & 0 \end{array}\right).$$
Otherwise, using the fact that ${^t M}=\ov{M}$, we obtain
\begin{eqnarray}
T &=& \left(\begin{array}{cc} |\omega_1|^2 \im H & \re(\omega_1 \ov{\omega_2}) \im H + \im(\omega_1 \ov{\omega_2}) \re H    \\
\re(\omega_1 \ov{\omega_2}) \im H - \im(\omega_1 \ov{\omega_2}) \re H &  |\omega_2|^2 \im H
\end{array}\right) \nonumber \\
&=& \left(\begin{array}{cc} \frac{-\im {^t M} }{\im 1/\tau} &  \frac{\re \tau}{\im \tau} \im{^t M} + \re{^t M} \\
 \frac{\re \tau}{\im \tau} \im{^t M} - \re{^t M} & \frac{\im{^t M} }{\im \tau} \end{array}\right) \nonumber\\
 &=&  \left(\begin{array}{cc} \frac{\im M}{\im 1/\tau} &  \re M - \frac{\re \tau}{\im \tau} \im M \\
 -{^t \left( \re{ M} - \frac{\re \tau}{\im \tau} \im{ M}\right)} &  -\frac{\im{ M} }{\im \tau} \end{array}\right). \label{Aform}
\end{eqnarray}

As $T$ is a non degenerate alternating form associated to the principal polarization $a$, there is a matrix $B \in \sM_{2g}(\Z)$ such that $B T {^t B}=J_{2g}$. $B$ is defined up to the action of the symplectic group $\Sp_{2g}(\Z)$. The columns of ${^t B} \Z^{2g}$ form a symplectic basis for the polarization $a$. Hence $\Omega_a=\Omega_0 {^t B}$ is  a period matrix associated to $a$ and if we write $\Omega_a=[\Omega_1,\Omega_2]$ with $\Omega_i \in M_g(\C)$, $\tau_a=\Omega_2^{-1} \Omega_1$ is a Riemann matrix associated to $a$.\\
Note that, if $\Omega_0=[Z_1,Z_2]$ with $Z_2$ invertible  and if,
 as usually,  we write $B=\left(\begin{array}{cc} \alpha & \beta \\ \gamma & \delta \end{array}\right)$
and $\tau_0=Z_2^{-1}Z_1$,  then
$\tau_a=B.\tau_0:=(\alpha \tau_0 +\beta)( \gamma \tau_0+\delta)^{-1}$.

\section{Explicit computations of $\chi_{18}$} \label{secexplicit}
We recall the hypothesis of Sec.\ref{Apower} and \ref{link}. Let $k$ be a number field and let $E : y^2+a_1 xy+a_3 y=f(x)$, $a_1,a_3 \in k, f\in k[x]$ be an elliptic curve with CM by an order  $\calO_K$ contained in an imaginary quadratic field $K=\Q(\sqrt{-d})$ with $d>0$ square free. We assume that all endomorphisms of $E$ are defined over $k$.
Let $\omega_E=dx/(2y+a_1x+a_3)$ be a regular differential on $E$. Prop.\ref{indecomp} shows that $A=E^3$ can be a Jacobian over $\C$ if there exists an indecomposable positive definite hermitian form $M \in \sM_3(\calO_K)$ of determinant $1$. For such an M, we want to compute
$$\chi:=\chi_{18}((A,a_0 M),\omega_0)$$
where $$\omega_0=p_1^*(\omega_E) \wedge p_2^*(\omega_E) \wedge p_3^*(\omega_E)$$
with $p_i: E^3 \to E$ the $i$th projection.\\

We will give examples in the case where $\calO_K$ is maximal and principal. This gives the possibilities $d=3,4,7,8,11,19,43,67,163$.  By \cite[p.418]{hoffmann}, there is no form $M$ as above if $d=3,4,8,11$. We then restrict ourselves to  cases which are left. 

\subsection{Choice of the models} \label{explicit}
Let us make the following elementary remark. If $u : E \to E'$ is an isomorphism such that $u^* \omega_{E'}=\alpha \omega_E$ then $\chi_{18}((E'^3,a'),\omega_0')=\alpha^{54} \chi_{18}((E^3,u^* a),\omega_0)$ with $\omega'_0=p_1^*(\omega_{E'}) \wedge p_2^*(\omega_{E'}) \wedge p_3^*(\omega_{E'})$. Hence we need to fix a precise model for our computations. \\
We will denote $E(d)$ the elliptic curve with CM by the maximal order of $\Q(\sqrt{-d})$ with minimal conductor, which is $d^2$ (see \cite[p.21]{gross}) and take $E=E(d)$. We recall the equations from \cite[p.82-84]{gross} in Table \ref{grossmodel}.\\
We use these models for two reasons. One is that, with a view to applications to optimal curves, we need a nice formula for the trace of $E$. As we can see from Lem.\ref{nbpoints}, the curves $E(d)$ have the simplest expression among all their twists. Another, more important, reason is that for this choice we expect the value of $\chi$ to be  in $\calO_K$ and  `minimal'. Indeed, since only $d$ divides the discriminant of $E(d)$, $E(d)$ is actually  an abelian scheme over $\Z[1/d]$, so the value of $\chi$ is in $\calO_K[1/d]$. Using a good arithmetic compactification  $\ov{\Am}_3$ of $\Am_3$, we see that $E(d)$ is a semi-abelian scheme, so defines a point of $\ov{\Am}_3 \otimes \Z$. Hence, the value of $\chi$ is actually in $\calO_K$. Also, as explained in Sec.\ref{hyperelliptic}, if a prime divides $\chi$, we can expect to have hyperelliptic or bad reduction. Thus, if a prime divides the discriminant of $E$, we conjecture that this prime also appears in $\chi$ (and this is always the case in our examples). Choosing $E(d)$, which has minimal conductor, we hope for a `minimal result'.

\begin{lemma}[{\cite[p.32]{gross}}] \label{nbpoints}
Let $d=7,19,43,67$ or $163$. Let $E(d)$ be the elliptic curve given in Table \ref{grossmodel}. The trace of Frobenius $\pi$ of $E(d) \otimes \GF_p$ where $p$ is a prime different from $d$ is
\begin{itemize}
\item $0$ if $\left(\frac{p}{d}\right)=-1$;
\item $a_p$ if $\left(\frac{p}{d}\right)=1$ where $a_p$ is the unique integer such that
$$4 p=a_p^2 + d b_p^2, \quad \left(\frac{2 a_p}{d}\right)=1.$$
\end{itemize}
\end{lemma}

\begin{table}
\caption{Gross' models $E(d)$}
\label{grossmodel}
\begin{tabular}{|c|c|c|}
\hline
$d$ & model & discriminant  \\
\hline
$7$ & $y^2+xy=x^3-x^2-2x-1$ & $-7^3$  \\
\hline
$19$ & $y^2+y=x^3-2 \cdot 19 x+ \frac{19^2-1}{4}$ & $-19^3$  \\
\hline
$43$ & $y^2+y=x^3-2^2 \cdot 5 \cdot 43 x + \frac{3 \cdot 7 \cdot 43^2-1}{4}$ & $-43^3$  \\
\hline
$67$ & $y^2+y=x^3-2 \cdot 5 \cdot 11 \cdot 67 x + \frac{7 \cdot 31 \cdot 67^2-1}{4}$ & $-67^3$  \\
\hline
$163$ & $y^2+y=x^3-2^2 \cdot 5 \cdot 23 \cdot 29 \cdot 163 x+ \frac{7 \cdot 11 \cdot 19 \cdot 127 \cdot 163^2-1}{4}$ & $-163^3$  \\
\hline
\end{tabular}
\end{table}

\subsection{Details of the computations in the case $d=7$} \label{detail7}

Let $\delta_1,\delta_2$ be generators of $H_1(E(7) \otimes \C,\Z)$. We can choose these generators such that the periods $[\omega_1,\omega_2]=[\int_{\delta_1} dx/(2y+x),\int_{\delta_2} dx/(2y+x)]$ satisfy $\omega_1/\omega_2 \in \H_1$ (here we take $\omega_1/\omega_2=\tau:=(1+i \sqrt{7})/2$). According to the data in \cite{schiemann},  $$M=\left(\begin{array}{ccc} 2 & 1 & 1 \\ 1 & 2 & \ov{\tau} \\ 1 & \tau & 2 \end{array}\right)$$
is --up to isomorphisms-- the unique indecomposable positive definite hermitian  form of determinant $1$ in  $\sM_3(\calO_K)$. Using \eqref{Aform} we get
$$T=  \left(\begin{array}{cccccc} 0 & 0 & 0 & 2 & 1 & 1 \\
0 & 0 & 2 & 1 & 2 & 1 \\
0 & -2 & 0 & 1 & 0 & 2 \\
-2 & -1 & -1 & 0 & 0 & 0 \\
-1 & -2 & 0 & 0 & 0 & 1 \\
-1 & -1 & -2 & 0 & -1 & 0 \\
\end{array}\right).$$
We can now ask MAGMA what is the Frobenius form of this matrix
$$J,B:=\texttt{FrobeniusForm}(T)=\left(\begin{array}{cccccc}   0&  0&  0&  1&  0&  0 \\
 0&  0&  0&  0&  1&  0\\
 0&  0&  0&  0&  0&  1\\
-1&  0&  0&  0&  0&  0\\
 0& -1&  0&  0&  0&  0\\
 0&  0& -1&  0&  0&  0
 \end{array}\right),
 \quad
\left(\begin{array}{cccccc}   0  & 1 & 0 & 0 & 0 & 0\\
1 & 0& -2&  4&  0&  0\\
 1&  0& -2& -3&  3&  1\\
 0&  0&  0&  1&  0&  0\\
 0&  0&  0 &-2&  1&  0\\
 2& -1& -3& -2 & 4 & 0
\end{array}\right)
 $$
where $B$ is a matrix such that $B T {^t B}=J=J_6$. Thus with the notation of Sec.\ref{link}, if $a=a_0 M$, then $\Omega_a=[\Omega_1,\Omega_2]=\Omega_0 {^t B}$ and $$\tau_a=\left(\begin{array}{ccc}
2 \tau & \tau &0 \\ \tau & -3+2\tau/3 & 2/3 \\ 0 & 2/3 & \tau/6
 \end{array}\right).$$
The Riemann matrix $\tau_a$  offers bad convergence for the computation of Thetanullwerte. In order to speed it up, we use the MAPLE function $\texttt{Siegel}(\tau_a)$ (see \cite{deconinck}). It returns a Riemann matrix $\tau'$ and $B' \in \Sp_6(\Z)$ such that $\tau'=B'.\tau_a$. We then let 
$$\Omega'=[\Omega'_1,\Omega'_2]=\Omega_a {^t B'}.$$
Again, we use  MAPLE to compute the $36$ even Thetanullwerte with a given precision and  then an approximation of  the expression
$$\chi:=  \pi^{54} \cdot 2^{26} \cdot \frac{\chi_{18}(\tau')}{(\det \Omega'_2)^{18}}.$$
Doing this with $50$ digits of precision, we get
$\chi=(7^7)^2.$

\begin{remark} \label{kleinquartic}
We recognize the square of the discriminant of the Klein quartic $X_{7,1} : x^3y+y^3z+z^3x=0$ (see \cite[Sec.3]{LRZ} for a definition of the canonical discriminant). This is an example of Klein's formula (see \cite[Th.4.1.2]{LRZ}) and it is no surprise as  it is known that $\Jac(X_{7,1}) \simeq_{K} (E(7)^3,a_0 M)$.  
\end{remark}

\subsection{The tables} \label{tablesec}
In Tables \ref{others} and \ref{others2}, we gather the results we found  when $d=7,19,43,67$, $A=E(d)^3 \otimes \Q$ and $M$ runs over the  indecomposable positive definite hermitian forms of dimension $3$ and determinant $1$ with coefficients in $\Z[\tau]$ with  $\tau=(1+\sqrt{-d})/2$, up to isomorphisms (see \cite{schiemann}). As there are more than 100 possibilities for $d=163$ we include only two cases, one with trivial automorphism group and the other with the only automorphism group of order $12$. Let us give some precisions on the contents of these tables.
\begin{enumerate}
\item the first column contains the discriminant $d$ and the index of $M$ in the web tables of \cite{schiemann}.
\item $\chi_{18}((A,a_0 M),\omega_0)$ is computed as described in Sec.\ref{explicit} and Sec.\ref{detail7}. Note that since $(A,a_0 M)$ is defined over $\Z[\tau]$, this value may not be in $\Z$. In the tables we denote by $[a,b]$  the element $a+b \tau$. Also the second line gives the primes below $a+b \tau$ in the same order. For all these irrational cases, we find that there is a form in \cite{schiemann} for which we obtain the conjugate value. This form could have been taken as $\ov{M}$ but it is not the case in \textit{loc. cit.} Thus the first column contains a second index for the form corresponding to the conjugate value.
\item the third column gives the order of the automorphism group $G$ of $M$, i.e. the set of $Q \in \GL_3(\Z[\tau])$ such that ${^t \ov{Q}} M Q=M$. It is easy to see that such an automorphism $Q$ is actually an automorphism of the principally polarized abelian variety $(A,a)$. When $(A,a)$ is geometrically the Jacobian of a non hyperelliptic curve $C$, Torelli's theorem shows that the order of the automorphism group of $C$ is half the order  of $G$.
\end{enumerate}

We now give a dual point of view which was suggested to us by Serre. Let $E_{d,i}$ be the quadratic twist of $E(d)$ by the  the non square part $\delta$ of the value of $\chi$ corresponding to the form $M$ in the line $d,\# i$ of our tables.  By the remark at the beginning of Sec.\ref{explicit}, we see that  $\chi_{18}((E_{d,i}^3,a_0 M),\omega_0)=\delta^{27} \cdot \chi$ is now a square in  $K=\Q(\sqrt{-d})$, so there is a genus $3$ curve $X_{d,i}$ defined over $K$ such that $\Jac(X_{d,i}) \simeq E_{d,i}^3$.

\begin{example}
Let us consider the case $19, \#1$. The quadratic twist by $-2$ of $E(19)$ is 
$$E_{19,1} : y^2=x^3-152x-722.$$
Using the construction of \cite{guardia} applied to our period matrix, it is possible to compute a model $f=0$ for the curve $X_{19,1}$. Guardia found
{\small \begin{eqnarray*}
f &= & U^4+2 U^3 V-2 U^3 W+\left(6-3 i \sqrt{19}\right) U^2 V^2 
 +18 U^2 V W+\left(6+3 i \sqrt{19}\right) U^2 W^2 \\
 & & +\left(5-3 i
  \sqrt{19}\right) U V^3 
 +\left(15+3 i \sqrt{19}\right) U V^2 W+\left(-15+3 i \sqrt{19}\right) U V W^2 \\
 & & +\left(-5-3 i
  \sqrt{19}\right) U W^3 
 +\frac{1}{2} \left(3-3 i \sqrt{19}\right) V^4+\left(12+4 i \sqrt{19}\right) V^3 W-30 V^2
  W^2 \\
  & & +\left(12-4 i \sqrt{19}\right) V W^3+\frac{1}{2} \left(3+3 i \sqrt{19}\right) W^4=0.
  \end{eqnarray*}}
 One can check that the discriminant of $-f/2$ is $2^{19} \cdot 19^7$ which is indeed a square root of $\chi_{18}((E_{19,1}^3,a_0 M),\omega_0)=(-2)^{27} \cdot (2^5 \cdot 19^7)^2 \cdot (-2)$.
Note that this curve can be descended over $\Q$ and after some simplifications, one finds
 $$X_{19,1} : x^4+(1/9) y^4+(2/3) x^2 y^2-190 y^2-570 x^2+(152/9) y^3-152 x^2 y-1083=0.$$
\end{example}

\subsection{The question of hyperelliptic reduction} \label{hyperelliptic}

The knowledge of the value of $\chi$ is not sufficient to understand if the reduction modulo a prime dividing $\chi$ gives a hyperelliptic curve or a decomposable polarization. In order to make it more precise, one has to study the expression
$$\Sigma_0=(2i \pi)^{140} \cdot \frac{\widetilde{\Sigma}_{140}(\tau_a)}{\det(\Omega_2)^{140}} (\omega_1 \wedge \omega_2 \wedge \omega_3)^{\otimes 140} \in \bfS_{3,140}(\Z)$$
where $\widetilde{\Sigma}_{140} \in \bfS_{3,140}(\C)$ is the modular form defined
by the thirty-fifth elementary symmetric function of the eighth power of the
even Thetanullwerte.  According to Igusa \cite{igusa2}, for $\tau \in \H_3$, and the principal polarization $T=J_6$ on $A_{\tau}=\C^3/\tau \Z^3+\Z^3$ we have that
\begin{itemize}
\item $(A_{\tau},J_6)$ is a non hyperelliptic Jacobian if and only if $\widetilde{\chi}_{18}(\tau) \ne 0$;
\item $(A_{\tau},J_6)$ is a hyperelliptic Jacobian if and only if $\widetilde{\chi}_{18}(\tau) = 0$
and $\widetilde{\Sigma}_{140}(\tau) \neq 0$;
\item $(A_{\tau},J_6)$ is decomposable if and only if $\widetilde{\chi}_{18}(\tau) =
\widetilde{\Sigma}_{140}(\tau) = 0.$ 
\end{itemize}
We need a version of this theorem which is true over any field. First we have to make the  form $\Sigma_0$ primitive.
\begin{lemma} \label{primitive}
The form $$\Sigma_{140}=\frac{(2i \pi)^{140}}{2^{218}} \cdot \frac{\widetilde{\Sigma}_{140}(\tau_a)}{\det(\Omega_2)^{140}} (\omega_1 \wedge \omega_2 \wedge \omega_3)^{\otimes 140}$$  is  a primitive Siegel modular form of weight $140$ over $\Z$.
\end{lemma}
\begin{proof}
To prove this, we study the first Fourier coefficient of $\widetilde{\Sigma}_{140}$ as Ichikawa did in  \cite{Ichi3}
for $\widetilde{\chi}_{18}$. Let $[\epsm]= \Ch{\varepsilon_{1}}{\varepsilon_{2}} $ be an  even characteristic. If
 $\varepsilon_1 = 0$  then $\theta [\epsm](0, \tau) = 1+ 2 \cdots$. If $\varepsilon_1 \ne  0$, 
$\frac{1}{2} \theta [\epsm](0, \tau)$ has integer Fourier coefficients.
Since for each $[\epsm_0]$ even 
$$\prod_{[\epsm] \ne [\epsm_0], \,  \textrm{even}} \theta[\epsm](0, \tau)^8$$
has a maximum of $8$ Thetanullewerte such that $\varepsilon_1=0$, $\widetilde{\Sigma}_{140}(\tau)$ divided by $c=2^{8 \cdot (35-8)}=216$ has still integer Fourier coefficients. On the other hand computing $c^{-1} \Sigma_0$ in the cases $d=7$, $d=19$ and $d=43$, \# 1 we find respectively
$$\begin{array}{l}
  2^{2} \cdot 3^3 \cdot 5 \cdot 7^{105} \cdot 13 \cdot 67,   \\ 
 2^{92} \cdot 3^3 \cdot 19^{105} \cdot 29 \cdot 31,   \\
 -2^{94} \cdot 3^3 \cdot 5 \cdot 43^{105} \cdot 827 \cdot 888001 \cdot 2458861813949 \cdot   \\ 96551756361358517199893386077757285219636855141244663. \end{array}$$
So only $2$ and $3$ can still divide the form $c^{-1} \Sigma_0$.  In order to see if this is the case, we compute the first Fourier coefficient of $c^{-1} \widetilde{\Sigma}_{140}$ using MAGMA. Unfortunately this computation is too large to be done over $\Z$ so we do a modular computation. It appears that $4$ divides the first Fourier coefficient but not $8$ or $3$. Since $\theta  \Ch{0}{\varepsilon_{2}}(0, \tau) = 1+ 2 \cdots$ one can also see that $4$ divides the form $c^{-1} \Sigma_0$ and that the form $\Sigma_0$ divided by $2^{216+2}$ is now primitive.
\end{proof}

Hence we have two primitive Siegel modular forms defined over $\Z$ :
\begin{equation*} 
\chi_{18}((A,a))=\frac{(2\pi)^{54}}{2^{28}} \cdot  \frac{\widetilde{\chi}_{18}(\tau_a)}{\det(\Omega_2)^{18}} (\omega_1 \wedge \omega_2 \wedge \omega_3)^{\otimes 18} \in \bfS_{3,18}(\Z)
\end{equation*}
and
\begin{equation*} 
\Sigma_{140}((A,a))=\frac{(2\pi)^{140}}{2^{208}} \cdot  \frac{\widetilde{\Sigma}_{140}(\tau_a)}{\det(\Omega_2)^{140}} (\omega_1 \wedge \omega_2 \wedge \omega_3)^{\otimes 140} \in \bfS_{3,140}(\Z).
\end{equation*}

Let $(A,a)$ be a principally polarized abelian variety defined over a field $k$.  In the proof of Prop.\ref{chi18F} we have already seen   that   $(A,a)$ is a non hyperelliptic Jacobian if and only if $\chi_{18}((A,a)) \ne 0$ and that if $(A,a)$ is a hyperelliptic Jacobian or decomposable then $\chi_{18}((A,a)) = 0$.   Now, since $\widetilde{\Sigma}_{140}$ is zero on $(\Am_3 \setminus t(\Mm_3)) \otimes \C$, it is zero on the schematic closure. So if $(A,a)$ is decomposable then $\Sigma_{140}(A,a)=0$. This proves
\begin{itemize}
\item if $\Sigma_{140}((A,a)) \ne 0$ and $\chi_{18}((A,a))=0$ then $(A,a)$ is the Jacobian of a hyperelliptic curve;
\item  if $(A,a)$ is decomposable then 
$\chi_{18}((A,a)) =\Sigma_{140}((A,a)) = 0.$ 
\end{itemize}
To complete the equivalence relations, we need only to see that if $\chi_{18}((A,a)) =\Sigma_{140}((A,a)) = 0$ then $(A,a)$ is decomposable. Since $\chi_{18}$ and $\Sigma_{140}$ are primitive, the corresponding divisors have no `vertical component' which is a necessary condition for the equivalence relations to hold. However it is not completely clear to us how these divisors intersect and we only  conjecture that they do it in a `nice' way.
\begin{conjecture} \label{conj}
If $\chi_{18}((A,a)) =\Sigma_{140}((A,a)) = 0$ then $(A,a)$ is decomposable.
\end{conjecture}

\begin{corollary}
Let $(A,a)$ be a principally polarized abelian variety defined over a field $k$.  Under conjecture \ref{conj}, we have that 
\begin{itemize}
\item $(A,a)$ is a non hyperelliptic Jacobian if and only if $\chi_{18}((A,a)) \ne 0$;
\item $(A,a)$ is a hyperelliptic Jacobian if and only if $\chi_{18}((A,a)) = 0$
and $\Sigma_{140}((A,a)) \neq 0$;
\item $(A,a)$ is decomposable if and only if $$\chi_{18}((A,a)) =
\Sigma_{140}((A,a)) = 0.$$ 
\end{itemize}
\end{corollary}

\begin{example}
Let $\Sigma=\Sigma_{140}((A,a_0 M),\omega_0)$ defined in the same way as $\chi$. We use the data described in the proof of Lem.\ref{primitive}. In the case $d=43$, \#1  we see that the prime $47$ divides $\chi$ but not $\Sigma$. For such a prime, we see that the non hyperelliptic Jacobian $(A,a)$ reduces modulo $47$ to the Jacobian of a genus $3$ hyperelliptic curve. Consider now the case $d=7$ for which we know that $(A,a)$ is the Jacobian of the Klein quartic $X_{7,1}$ by Rem.\ref{kleinquartic}. The prime $7$ divides both $\chi$ and $\Sigma$. Thus, $X_{7,1}$ has bad reduction at $7$. However, as explained in \cite{elkies}, the Klein quartic has potentially good reduction at $7$ and its reduction is the hyperelliptic curve $H_{7,1} : y^2=x^7-x$. Indeed, over a number field which contains $(-7)^{1/4}$, the elliptic curve $E(7)$ can be transformed into the elliptic curve $E' : y_1^2=4 x_1^2-3/4 \cdot \sqrt{-7} x_1^2-x_1$, which has good reduction at $7$, by $$x_1=\frac{x-2}{\sqrt{-7}}, \quad y_1=\frac{2y+x}{2 (-7)^{3/4}}.$$
We get that 
$$\omega_E= \frac{dx}{2y+x}=(-7)^{1/4} \frac{dx_1}{2y_1}=(-7)^{1/4} \omega_{E'}.$$
Thus the value of $\Sigma$ is changed by $((-7)^{1/4})^{-140 \cdot 3}=(-7)^{-105}$ which canceled the factor $7^{105}$ in $\Sigma$ whereas $\chi$ becomes $7^{14} \cdot ((-7)^{1/4})^{-18 \cdot 3}=(-7)^{1/2}$. The computations agree with the theory.  A similar computation can be done for the curve $X_{19,1}$. For $p=19$, the curve  $X_{19,1}$has potentially good reduction and this reduction is the hyperelliptic curve $H_{19,1} : y^2=x^7-x$. For $p=2$, from the valuations of $\chi$ and $\Sigma$ and under Conj.\ref{conj} we can deduce that $X_{19,1}$ has not, even potentially, good reduction.
\end{example}

\subsection{Application to optimal genus $3$ curves}
Let $p=47$. We want to prove that there is a genus $3$ curve $C$ over $F=\GF_{47}$ whose number of rational points reaches Serre-Weil bound, i.e. $\# C(F)=p+1+3 \lfloor 2\sqrt{p} \rfloor=87$.
According to \cite{lauter}, if such a curve exists,  its Jacobian is isogenous to the third power of an elliptic curve $E/F$ with trace $- \lfloor 2\sqrt{p} \rfloor=-13$. This means that the elliptic curve $E$ is ordinary and thus such a curve exists. Moreover  if $\pi$ denotes the Frobenius endomorphism of $E$ then $\Z[\pi] \simeq \Z[(13+\sqrt{13^2-4\cdot 47})/2]=\Z[\tau]$ where $\tau=(1+\sqrt{-19})/2$. So, $\End(E)=\Z[\pi]$ is the ring of integers $\calO_K$ of $K=\Q(\sqrt{-19})$. Note that, since $\left(\frac{-19}{47}\right)=1$, all endomorphisms are defined over $F$. Moreover since the class number of $\calO_K$ is $1$, $E$ is unique up to isomorphisms.\\
By Lem.\ref{onecurve} we know that $\Jac(C)$ is actually isomorphic to $E^3$. By Lem.\ref{inde=inde} the existence of an indecomposable principal polarization on $\Jac(C)$   translates into the existence of an indecomposable positive definite hermitian form $M \in \sM_3(\calO_K)$ of determinant $1$.  In \cite{schiemann}, such forms have been classified up to isomorphisms for some imaginary quadratic orders. In the present case, there exists only one, given by
$$M=\left(\begin{array}{ccc}
2 & 1 &-1 \\ 1 & 3 & -2+\tau\\ -1 & -2+\ov{\tau}&3
 \end{array}\right).$$
 So,  if $C$ exists then $\Jac(C) \simeq (E^3,a)$ with $a=a_0 M$.\\
 
Let us consider the quadratic twist $E'$ of $E$ and  $(A',a')=(E'^3,a_0 M)$ the quadratic twist of $(E^3,a)$. Let $E(19)$ be the model of the elliptic curve with CM by $\sqrt{-19}$ given in Table \ref{grossmodel}. Using Lem.\ref{nbpoints} we can see that $\textrm{Tr}(E(19) \otimes F)=13$, i.e. $E(19) \otimes F=E'$. Since $\End(E(19) \otimes \Q)=\calO_K$,  we  consider the principally polarized abelian variety $(\tilde{A},\tilde{a})=(E(19)^3,a_0 M)$ which is a lift of $(A',a')$. Table \ref{others} shows that
$$\chi_{18}((\tilde{A},\tilde{a}) \otimes \Q,\omega_0 \otimes \Q)=(2^5 \cdot 19^7)^2 \cdot (-2).$$ 
It is not a square over $\Q$ and since $\left(\frac{-2}{47}\right)=-1$ it is a  non-square over $F$. Hence Prop.\ref{mainresult} shows that $(A',a')$ is not a Jacobian so by Th.\ref{twist} its quadratic twist $(A,a)$ is. In conclusion, we obtained that there exists an optimal curve over $\GF_{47}$ but no minimal curve (i.e. whose number of rational points is equal to $p+1-3 \lfloor 2\sqrt{p} \rfloor=9$).\\

In the same way, we can prove that there is an optimal but no minimal genus $3$ curve over $\GF_p$ for 
$p=61,137,277$ and a minimal but no optimal genus $3$ curve  for $p=311$. Note that a model for each of these curves can be obtained by reducing modulo $p$ the model of $X_{19,1}$ given in Sec.\ref{tablesec}.

Also, using $d=67$ and respectively the form $\# 3$ and $\# 7$ we can prove that there is a minimal and an optimal curve over $\GF_{23^3}$.

\begin{corollary} \label{someoptimal}
We have $N_q(3)=q+1+3 \lfloor 2 \sqrt{q} \rfloor$ for $q=47,61,137,277$ and $23^3$. 
\end{corollary}

\begin{remark}
Some of these values have been confirmed by explicit computations by Top \cite{top} and  Alekseenko et al. \cite{zait}.
\end{remark}

 \begin{landscape} \begin{table} \caption{Computation of $\chi$} \label{others}   \begin{tabular}{|c|c|c|c|}
\hline
$d$ & $M$ & 
$\chi:=\chi_{18}((A,a_0 M),\omega_0)$ & $\# \Aut(A,a_0 M)$ \\
\hline
$7$, \# 1 &  $\left(\begin{array}{ccc} 2 & 1 & 1 \\ 1 & 2 & \ov{\tau} \\ 1 & \tau & 2 \end{array}\right)$ &
 $(7^7)^2$ 
& $2 \cdot 168$ \\
 \hline
 $19$  \# 1 & $\left(\begin{array}{ccc}
2 & 1 &-1 \\ 1 & 3 & -2+\tau\\ -1 & -2+\ov{\tau}&3
 \end{array}\right)$ & 
 $(2^5 \cdot 19^7)^2 \cdot (-2)$ 
& $2 \cdot 6$ \\
 \hline
 $43$,  \# 1 &  $\left(\begin{array}{ccc} 3 & 1 & 1-\ov{\tau} \\ 1 & 4 & 2 \\ 1-\tau & 2 & 5 \end{array}\right)$ &
 $(2^6 \cdot 43^7)^2 \cdot (-47 \cdot 79 \cdot 107 \cdot 173)$ 
 & $2 \cdot 1$ \\
 \hline
 $43$,  \# 2&  $\left(\begin{array}{ccc} 3 & 1+\ov{\tau} & 2-\ov{\tau} \\ 1+\tau & 5 & 2-\ov{\tau} \\ 2-\tau & -2-\tau & 5 \end{array}\right)$ &
 $(2^5 \cdot 3^4 \cdot 43^7)^2 \cdot (-2 \cdot 3 \cdot 7)$ & $2 \cdot 6$ \\
 \hline
  $43$, \# 3 &  $\left(\begin{array}{ccc} 2 & -1 & 1 \\  -1 & 4 & 1+\ov{\tau} \\ 1 & 1-\tau & 4 \end{array}\right)$ &
 $(2^6 \cdot 5^3 \cdot 43^7)^2 \cdot (-487)$ & $2 \cdot 2$ \\
 \hline
  $43$,  \# 4,5 &  $\left(\begin{array}{ccc} 3 & 1 & -1-\ov{\tau} \\  1 & 3 & -1 \\ -1-\tau & -1 & 5 \end{array}\right)$ &
 \begin{tabular}{c}
 $-2^{11} \cdot 3^9 \cdot [1,2]^{27}  \cdot [5,-2]  \cdot [7,-2] \cdot [17,-4]= $ \\
 $-2^{11} \cdot 3^9 \cdot 43^{13} \cdot \left(43 \cdot 59 \cdot 79 \cdot 397\right)^{1/2}$ \end{tabular}
  & $2 \cdot 2$ \\
 \hline
  $67$,  \# 1,2 &  $\left(\begin{array}{ccc} 5 & -1-\ov{\tau} & -\ov{\tau} \\  -1-\tau & 5 & 2  \\ -\tau & 2 & 6 \end{array}\right)$ &
\begin{tabular}{c} $2^{11}  \cdot [-1,2]^{28} 
 \cdot  [1,2] \cdot  [-11,2]  \cdot  [-15,2] \cdot [3,4] \cdot$ \\ $[1,6 ] \cdot [23,2]\cdot [21,4]  \cdot [-49,2]   \cdot [43,6] \cdot [55,6] \cdot [53,16]= $ \\
 $2^{11} \cdot 67^{14} \cdot (71 \cdot 167 \cdot 263 \cdot 293 \cdot $\\ $619 \cdot 643 \cdot 797 \cdot 2371 \cdot 2719 \cdot 3967 \cdot 8009)^{1/2}$
 \end{tabular} & $2 \cdot 1$ \\
 \hline
 $67$,  \# 3  &  $\left(\begin{array}{ccc} 5 & -2+\ov{\tau} & -1-\ov{\tau} \\  -2+\tau & 6 & -2  \\ -1-\tau & -2 & 7 \end{array}\right)$ & $(2^{6} \cdot 3^{6} \cdot 67^7)^2 (-13 \cdot 53 \cdot 71 \cdot 131 \cdot 3319)$ & $2 \cdot 1$ \\ \hline 
 $67$,  \# 4,5 &  $\left(\begin{array}{ccc} 3 & -1 & 1 \\  -1 & 4 & -\ov{\tau}  \\ 1 & -\tau & 5 \end{array}\right)$ &
\begin{tabular}{c}  $2^{12} \cdot 3^9 \cdot [-1,2]^{27} \cdot [-5,2] \cdot [-7,2] \cdot [-3,4] \cdot$ \\  $[1,4] \cdot [-13,4] \cdot [-15,8] \cdot [-33,8] \cdot [23,10] = $ \\
$2^{12} \cdot 3^9 \cdot 67^{13} \cdot (67 \cdot 83 \cdot 103 \cdot 269 \cdot$ \\ $277 \cdot 389 \cdot 1193 \cdot 1913 \cdot 2459)^{1/2}$ \end{tabular} & $2 \cdot 1$ \\
 \hline
\end{tabular}\end{table}\end{landscape}

 \begin{landscape} \begin{table} \caption{Computation of $\chi$} \label{others2}   \begin{tabular}{|c|c|c|c|}
\hline
$d$ & $M$ & 
$\chi:=\chi_{18}((A,a_0 M),\omega_0)$ & $\# \Aut(A,a_0 M)$ \\
\hline
$67$,  \# 6 &  $\left(\begin{array}{ccc} 5 & -1+\ov{\tau} & \ov{\tau} \\  -1+\tau & 5 & 2  \\ \tau & 2 & 5 \end{array}\right)$ & $(2^6 \cdot 5^3 \cdot 67^7)^2 \cdot 83 \cdot 211 \cdot 1637 \cdot 2441$ & $2 \cdot 1$ \\
\hline
$67$,  \# 7 &  $\left(\begin{array}{ccc} 2 & 0 & -1 \\  0 & 3 & -2+\ov{\tau}  \\ -1 & -2+\tau & 7 \end{array}\right)$ &
$(2^5 \cdot 7^4 \cdot 67^{7})^2 \cdot (-2 \cdot 7 \cdot 31)$ & $2 \cdot 6$ \\
\hline 
$67$,  \# 8,9 &  $\left(\begin{array}{ccc} 3 & -1 & -2+\ov{\tau} \\  -1 & 4 & 0  \\ -2+\tau & 0 & 7 \end{array}\right)$ & \begin{tabular}{c}
$2^{11} \cdot 7^6 \cdot [-1,2]^{27} \cdot [1,2]^2 \cdot [-37,6] \cdot [-71,12] \cdot [-71,30] = $ \\ $2^{11} \cdot 7^6 \cdot 67^{13} \cdot (67 \cdot 71^2 \cdot 1759 \cdot 6637 \cdot 18211)^{1/2}$ \end{tabular} & $2 \cdot 2$ \\ \hline 
$67$,  \# 10,12 &  $\left(\begin{array}{ccc} 2 & -1 & 0 \\  -1 & 4 & -1+\ov{\tau}  \\ 0 & -1+\tau  & 5 \end{array}\right)$ & \begin{tabular}{c}
$-2^{11} \cdot 3^9 \cdot 5^6 \cdot [-1,2]^{27} \cdot [-3,2] \cdot [3,2] \cdot [-21,2] \cdot [-31,2]=$\\
$-2^{11} \cdot 3^9 \cdot 5^6 \cdot 67^{13} \cdot (67 \cdot 71 \cdot 83 \cdot 467 \cdot 967)^{1/2}$ \end{tabular} & $2 \cdot 2$ \\
\hline
$67$,  \# 11 &  $\left(\begin{array}{ccc} 5 & \ov{\tau} & -2 \\  \tau & 6 & 2+\ov{\tau}  \\ -2 & 2+\tau  & 6 \end{array}\right)$ & $( 2^6 \cdot 3^4 \cdot  5^3 \cdot 67^7)^2 \cdot (-3 \cdot 7 \cdot 8731)$ & $2 \cdot 2$ \\
\hline
$67$,  \# 13 &  $\left(\begin{array}{ccc} 3 & 1 & -1 \\ 1 & 5 & -3+\ov{\tau}  \\ -1 & -3+\tau  & 5 \end{array}\right)$ &
$(2^8\cdot 5^4 \cdot 67^7)^2 \cdot (-2\cdot 5 \cdot 9769)$ & $2 \cdot 2$ \\
\hline  
$163$,  \# 3,4 &  $\left(\begin{array}{ccc} 7 & 3-\ov{\tau} & 2+\ov{\tau} \\ 3-\tau &82 & -3+\ov{\tau}  \\ 2+\tau & -3+\tau  & 14 \end{array}\right)$ &
\begin{tabular}{c}
$-2^{12} \cdot [-1,2]^{27} \cdot [-5,2] \cdot [15,4] \cdot [31,2] \cdot [67,8] \cdot [-137,8] \cdot$ \\  $[-39,28] \cdot [-49,44] \cdot [-743,94] \cdot [-169,164] \cdot [-907,158] \cdot [445,406] \cdot $ \\ $[-2507,342] \cdot [-3029,244] \cdot [-2777,388] \cdot [4043,74]=$ \\
$ -2^{12}  \cdot  63^{27} \cdot (179 \cdot 941\cdot 1187\cdot 7649\cdot 20297 \cdot$ \\$ 32573\cdot  79621\cdot 844483\cdot 1103581\cdot 1702867\cdot 7136971\cdot 10223179  \cdot $\\ $10876741\cdot 12806557\cdot 16869547)^{1/2}$ \end{tabular} & $2 \cdot 1$ \\
\hline 
$163$,  \# 85 &  $\left(\begin{array}{ccc} 2 & 1 & -\ov{\tau} \\ 1 & 2 & 1-\ov{\tau}  \\ -\tau & 1-\tau  & 28 \end{array}\right)$ &
$(2^5\cdot 7^4 \cdot 11^4 \cdot 163^7)^2 \cdot (-2\cdot 7 \cdot 11 \cdot 19 \cdot 127)$ & $2 \cdot 6$ \\
\hline  
\end{tabular}\end{table}\end{landscape}

\end{document}